\documentclass{amsart}

\usepackage{amsmath, amsthm}
\usepackage{amssymb}
\usepackage{amsfonts}
\usepackage{latexsym}
\usepackage{mathrsfs}
\usepackage{bm}
\usepackage{graphicx}

\theoremstyle{plain}
\newtheorem{theorem}{Theorem}[section]

\newtheorem{lem}[theorem]{Lemma}
\newtheorem{pro}[theorem]{Proposition}
\newtheorem{Def}[theorem]{Definition}

\numberwithin{equation}{section}

\newcommand{\ts}{tangent space}

\newcommand{\hym}{hyperbolic metric}

\newcommand{\nb}{normal bundle}
\newcommand{\scv}{solution curve}
\newcommand{\WPm}{Weil-Petersson metric}
\newcommand{\TS}{Teichm\"{u}ller space}

\newcommand{\kg}{Kleinian group}

\newcommand{\mc}{mean curvature}

\newcommand{\pc}{principal curvature}

\newcommand{\maxp}{maximum principle}

\newcommand{\hqd}{holomorphic quadratic differential}
\newcommand{\RS}{Riemann surface}

\newcommand{\sff}{second fundamental form}
\newcommand{\gse}{Gauss equation}
\newcommand{\af}{almost Fuchsian}

\newcommand{\qf}{quasi-Fuchsian}

\newcommand{\tm}{three-manifold}
\newcommand{\htm}{hyperbolic three-manifold}

\newcommand{\ms}{minimal surface}
\newcommand{\mi}{minimal immersion}
\newcommand{\cs}{conformal structure}
\newcommand{\cc}{conformal change}
\newcommand{\cvc}{convex core}

\newcommand{\is}{incompressible surface}
\newcommand{\lo}{linearized operator}
\newcommand{\dio}{the differential operator}

\newcommand{\C}{\mathbb{C}}
\newcommand{\R}{\mathbb{R}}
\newcommand{\I}{\mathbb{I}}
\renewcommand{\H}{\mathbb{H}}

\newcommand{\Fcal}{\mathcal{F}}

\DeclareMathOperator{\Iso}{Iso}

\DeclareMathOperator{\PSL}{PSL}

\numberwithin{equation}{section}

\allowdisplaybreaks

\def\XXint#1#2#3{{\setbox0=\hbox{$#1{#2#3}{\int}$}
    \vcenter{\hbox{$#2#3$}}\kern-.5\wd0}}

\makeatletter
\def\@citestyle{\m@th\upshape\mdseries}
\def\citeform#1{{\bfseries#1}}
\def\@cite#1#2{{%
  \@citestyle[\citeform{#1}\if@tempswa, #2\fi]}}
\@ifundefined{cite }{%
  \expandafter\let\csname cite \endcsname\cite
  \edef\cite{\@nx\protect\@xp\@nx\csname cite \endcsname}%
}{}
\makeatother

\begin{document}

\title
{Minimal Immersions of Closed Surfaces in Hyperbolic Three-Manifolds}

\author{Zheng Huang}
\address{Department of Mathematics,
The City University of New York,
Staten Island, NY 10314, USA.}
\email{zheng.huang@csi.cuny.edu}

\author{Marcello Lucia}
\address{Department of Mathematics,
The City University of New York,
Staten Island, NY 10314, USA.}
\email{mlucia@math.csi.cuny.edu}

\date{August 22, 2010}


\begin{abstract}
We study {\mi}s of closed surfaces (of genus $g \ge 2$) in hyperbolic {\tm}s, with prescribed data
$(\sigma, t\alpha)$, where $\sigma$ is a {\cs} on a topological surface $S$, and $\alpha dz^2$ is
a {\hqd} on the marked {\RS} $(S,\sigma)$. We show that, for each $t \in (0,\tau_0)$ for some
$\tau_0 > 0$, depending only on $(\sigma, \alpha)$, there are at least two {\mi}s of closed surface 
of prescribed {\sff} $Re(t\alpha)$ in the {\cs} $\sigma$. Moreover, for $t$ sufficiently large, there 
exists no such {\mi}. Asymptotically, as $t \to 0$, the {\pc}s of one {\mi}
tend to zero, while the intrinsic curvatures of the other blow up in magnitude.
\end{abstract}

\maketitle
\section {Introduction}
A fundamental problem in hyperbolic geometry is the interaction between the hyperbolic structures of
closed surfaces and those of {\tm}s. Minimal surface theory has been intimately related to the
geometry and topology of three-manifolds (see for instance \cite{SY79, MY82}). It follows from
\cite{FHS83} that any incompressible surface can be isotoped to a {\ms} in a closed Riemannian
three-manifold. Hence one expects any {\htm} to be obtained by gluing pieces of the type
$S \times (-a,a)$, for some $a > 0$, with $S$ minimal.

Minimal surfaces play important roles in understanding the structures of $3$-manifolds (see for
example \cite{SU82} and recent series of work by Colding-Minicozzi \cite{CM04a, CM04b, CM04c,
CM04d}). Closed surfaces (compact without boundary) cannot be minimally embedded in $\R^3$. In
the positive curvature case, the situation is quite different, since Lawson proved in \cite{LJ70} that every
compact orientable surface can be minimally immersed in the sphere $S^3$. The case of {\mi}s in
{\htm}s is much more subtle, and has been studied by several authors (see for example \cite{Tau04,
Rub05, Has05}). In particular, Uhlenbeck in~\cite{Uhl83} has undertaken a program to parametrize a
class of {\htm}s by incompressible minimal surfaces.

The goal of this paper is to investigate closed {\ms}s of genus at least two immersed in hyperbolic {\tm}s,
and to prove several results inspired by Uhlenbeck's approach. These surfaces admit {\hym}s via the
uniformization theorem, and the {\cc} function between the unique {\hym} on a marked surface
$(S,\sigma)$ and the induced metric from the immersion into a hyperbolic {\tm} satisfies the {\gse}, which
is a semilinear elliptic equation, and we study the solution curve to a family of {\gse}s and their
geometrical implications.

Throughout the paper, we assume $S$ as a closed oriented surface of genus $g \geq 2$. {\TS} of $S$ is
denoted by $T_g(S)$, and it is the space of {\cs}s (or equivalently {\hym}s) on $S$ such that two {\cs}s
$\sigma$ and $\rho$ are equivalent if there is an orientation-preserving diffeomorphism in the homotopy
class of the identity between them. Let $(S,\sigma)$ be the surface $S$ marked with the {\cs}
$\sigma \in T_g(S)$, and $z = x+iy$ be the conformal coordinates on $(S,\sigma)$, we denote the unique
{\hym} on $(S,\sigma)$ by $g_\sigma dzd\bar{z}$. When $(S,\sigma)$ is immersed in some {\htm} $M$,
its induced metric from the immersion by $f(z)dzd\bar{z}$ with
$h = h_{11}dx^2+2h_{12}dxdxy+h_{22}dy^2$ the {\sff}. Then it is well-known that (\cite{Hop89,LJ70}) the form
$\alpha = (h_{11} - ih_{12})dz^2$ is a {\hqd} and $h = Re(\alpha)$.

From the prescribed data $(\sigma,t\alpha)$, our goal is to construct {\htm}s such that the closed surface
$S$ is minimally immersed. In this aspect, our main result should be considered as a ``local realization
theorem". This {\mi} is governed by six equations: Three of them are in the form of curvature relation
since we require the {\nb} as a three-manifold is hyperbolic: $R_{i3j3} = -g_{ij}$ (see \eqref{ce}). They can
be reduced to a system of ODEs and they determine the metric explicitly in the {\nb} of the {\ms} $S$
(see \S 2.2); Two more relations are provided by the Codazzi equations: $R_{ijk3} =0$, and they ensure the requirement
for the prescribed {\sff} (also \S 2.2). The last equation is the {\gse} which ensures the {\mi}s stay in the
prescribed {\cs} $\sigma$. Using these equations, we build a {\htm}, topologically $S \times (-a,a)$, for
$a \in (0,\infty]$, around the {\ms} $S$. Among these equations, probably the {\gse} is the most intriguing.
The conformal factor between the induced metric $f(z)dzd\bar{z}$ and the {\hym} $g_\sigma dzd\bar{z}$
on $(S,\sigma)$ can be represented via $f(z) = e^{2u(z)}g_\sigma(z)$, and the {\gse} is an elliptic
semilinear equation given by
\begin{equation}\label{ge}
\Delta u + 1 -e^{2u} - \frac{|\alpha|^2}{g_\sigma^2}e^{-2u} = 0,
\end{equation}
where $\Delta$ is the Laplacian in the {\hym} $g_\sigma dzd\bar{z}$.

%
%

\begin{Def}
We call $S(\sigma,\alpha)$ a {\mi} with data $(\sigma,\alpha)$ if $S$ is marked by a {\cs} $\sigma \in T_g(S)$
and $S$ is a {\mi} whose {\sff} is given by $Re(\alpha)$, for $\alpha \in Q(\sigma)$.
\end{Def}

We consider a ray $t\alpha(z)dz^2$, for a fixed direction $\alpha \in Q(\sigma)$, and $t \ge 0$. Note that the
space of {\hqd}s, $Q(\sigma)$, on $(S,\sigma)$ is identified as the co{\ts} of {\TS} at the point
$\sigma \in T_g$, therefore $t\alpha$ represents a ray in $Q(\sigma)$, and this ray is closely related to the
notion of Teichm\"{u}ller geodesics in {\TS}. The data $(\sigma,\alpha)$ is a point in the cotangent bundle
$T_g(S) \times Q(\sigma)$, where $Q(\sigma)$ is a Banach space of real dimension $6g-6$. This ray
enables us to study the one-parameter family of {\gse}s:
\begin{equation}\label{t-ge}
\Delta u(t) + 1 -e^{2u(t)} - \frac{t^2|\alpha|^2}{g_\sigma^2}e^{-2u(t)} = 0,
\end{equation}
for {\mi}s $S(\sigma,t\alpha)$.

Using the implicit function theorem, Uhlenbeck (\cite{Uhl83}) proved the existence of a smooth solution curve to
the equation \eqref{t-ge}:
\vskip 0.1in
\noindent
{\bf Theorem} \cite {Uhl83}. {\it Fixing a {\cs} $\sigma \in T_g(S)$, and $\alpha \in Q(\sigma)$, there exists a
constant $\tau_0 > 0$, depending only on $(\sigma, \alpha)$, such that for each $t \in [0,\tau_0]$, there is a
stable {\mi} of $S$ with data $(\sigma,t\alpha)$ into some {\htm}.
}
\vskip 0.1in
Our main result in this paper is to obtain an additional solution for each Uhlenbeck's nonzero stable solution to 
the {\gse} in this paper, which can be formulated as the following theorem:

\begin{theorem} Let $S$ be a closed surface and $\sigma \in T_g(S)$ be a {\cs} on $S$. If
$\alpha \in Q(\sigma)$ is a {\hqd} on $(S,\sigma)$, then:
\begin{enumerate}
\item
for sufficiently large $t$, the {\gse} \eqref{t-ge} admits no solutions, i.e., there is no {\mi} of $S$ with data
$(\sigma,t\alpha)$ into some {\htm}s;
\item
there exists a constant $\tau_0 > 0$, such that, for each $t \in (0,\tau_0)$, there exist at least two {\mi}s of $S$
into some {\htm} in the conformal class of $\sigma$ with the {\sff} $Re(t\alpha)$.
\end{enumerate}
\end{theorem}

In Uhlenbeck's theorem, she also proved that there is a positive constant $\epsilon$, such that there is an
unstable solution on for each $t \in (\tau_0 - \epsilon,\tau_0)$. We point out that in this parameter interval 
$(\tau_0 - \epsilon,\tau_0)$, the {\mi} obtained from our additional solution might coincide with Uhlenbeck's 
unstable solution on the same interval.

The nature of these solutions indicates important geometric information on the {\ms}s, as well as the {\htm}s they
immerse into: At $t =0$, the surface $S$ is totally geodesic, and its {\nb} is a Fuchsian manifold, i.e., a warped
product {\htm}. For $t$ small enough along the Uhlenbeck {\scv}, and the {\pc}s of the {\mi} stay bounded in
magnitude less than $1$, its {\nb} is a so-called {\it {\af} manifold}, and $S$ is the unique {\ms} within this {\nb}
(\cite{Uhl83}). Then further along the solution curve, the {\mi}s remain stable until a particular parameter value,
but the {\nb} becomes finite. It will be very interesting to understand further solutions along this {\scv}, as well as
the geometry of {\htm}s when $t$ is approaching its maximal value when such a {\mi} is allowed.

In Theorem 1.2, we use Uhlenbeck's parameterization of the solution curve to study the equation \eqref{t-ge}, and
find an additional solution for each Uhlenbeck's stable solution along the {\scv}. We note that though these solutions
represent the same point in {\TS}, the {\htm}s they immersed into are quite different: by fixing the data $(\sigma, t\alpha)$,
we fix the conformal class of the surface and its {\sff} for the {\mi}. Different solutions represent different induced metrics
on the surface $(S, \sigma)$, hence two normal bundles are distinct. It is an intriguing question to ask how many {\mi}s
are allowed for each given data $(\sigma,t\alpha)$.

Naturally we also consider the asymptotic behavior of the solutions of the {\gse} \eqref{t-ge}. The blow-up analysis near
$t =0$ provides important structural information about the {\mi}s we obtained in Theorem 1.2:
\begin{theorem}
Let $\{u_n(t_n)\}$ be a sequence of solutions to the equation \eqref{t-ge} with $t_n \to 0$. Then, along a subsequence,
the following alternative holds
\begin{enumerate}
\item[(i)]
$u_n$ coincides with the solution obtained in Uhlenbeck's theorem, in which case, the normal bundle of $S$ is an {\af} {\tm},
and the {\pc}s are less than one in absolute value;
\item[(ii)]
 or
 $\|u_n\|_{\infty} \to \infty$, in which case, the absolute values of the intrinsic curvatures of the corresponding {\mi} go to
 infinity.
\end{enumerate}
\end{theorem}

Our technique is to study the variational theory for the solutions to the {\gse}. In calculus of variations, the problem of
obtaining additional solutions to some differential equation is well-studied (\cite{GT83}, \cite{Str00}): one rewrites the
equation such that the nonlinear operator is the derivative of an appropriate functional, and uses techniques such as the
mountain pass theorem from nonlinear functional analysis to find other critical points of the associated functional. Much
of the difficulty is that the usual variational setting for the problem does not satisfy the compactness property. We
introduce a different but equivalent inner product structure to the usual Sobolev space for the problem and prove the
mountain pass theorem in the new setting.
\subsection*{Plan of the paper}
We will collect preliminary results in section two. In particular, we briefly introduce hyperbolic geometry of dimensions
two and three, and set up the {\gse} in this setting. Section three is devoted to prove our main results, and it breaks into
several subsections: in $\S 3.1$, we prove a nonexistence theorem for large parameter $t$; in \S 3.2, we study
Uhlenbeck's solution curve and its parameterization; in $\S 3.3$, we work in the variational setting of the problem, and
define a new norm and show that the functional with the norm satisfies the Palais-Smale compactness condition in
$\S 3.4$, therefore develop the mountain pass structure of the solutions; and in the section $\S 3.5$, we prove the
Theorem 1.3.
\subsection*{Acknowledgements}
The research of Z.H. is partially supported by a PSC-CUNY award, while the research of M.L. is supported by projects
MTM2008-06349-C03-01 (Spain) and SGR2009-345 (Catalunya).

\section{Preliminaries}
\subsection{Hyperbolic surfaces and three-manifolds}
By hyperbolic spaces, we refer to Riemannian manifolds of constant sectional curvature $-1$. Naturally, a hyperbolic
manifold $M^n$ of dimension $n$ is a quotient space of $\H^n$ ($n \ge 2$), by a subgroup of the (orientation
preserving) isometry group $\Iso(\H^n)$. We only consider $n =2$ and $n=3$ in this paper. These two cases are
drastically different, largely due to Mostow's rigidity theorem.

In the case of $n =2$, we have $\Iso(\H^2) = \PSL(2,\R)$, which can be identified to a subgroup of $\Iso(\H^3) = \PSL(2,\C)$. Let
$S$ be a closed surface of genus $g \geq 2$, then there is a {\hym} in each {\cs} of $S$, by the uniformization theorem.
Let $\sigma$ be a {\cs} on $S$, it is a point in {\TS} $T_g(S)$. We often use $z$ and $g_{\sigma}dzd\bar{z}$ to
record the conformal coordinate and the {\hym} on $(S,\sigma)$. Similarly, we use $w$ and $g_{\rho}dwd\bar{w}$
to record the conformal coordinate and the {\hym} on another {\cs} $\rho \in T_g(S)$. The geometry of {\TS} is often
studied via its cotangent bundle. At $\sigma \in T_g(S)$, we have the cotangent space $Q(\sigma)$, where
$\alpha \in Q(\sigma)$ is a {\hqd} on $(S,\sigma)$. Locally, $\alpha = \alpha(z)dz^2$, where $\alpha(z)$ is
holomorphic.

Let $M^3 = \H^3/\Gamma$ be a hyperbolic three-manifold, and we assume $\Gamma \subset \PSL(2, \C)$ acts on
$\H^3$ properly and discontinuously. In this case we call $\Gamma$ a {\it {\kg}}. For any $p \in \H^3$, the orbit set of
$\Gamma$ has accumulation points on the boundary $S_{\infty}^2$. The closed set of these limit points is called the
{\it limit set} $\Lambda_{\Gamma}$ of the group $\Gamma$. There are two elementary types of {\htm}s we will
encounter frequently: when $\Lambda_{\Gamma}$ is a round circle, $M^3$ is called {\it Fuchsian}, which is a
product space of a hyperbolic surface $S$ and the real line $\R$. It is easy to see that the space of Fuchsian
manifolds is isometric to {\TS}; when $\Lambda_{\Gamma}$ lies in a Jordan curve, $M^3$ is called {\it {\qf}}, and it is
topologically $S\times \R$. In this case $M^3$ is a complete {\htm} quasi-isometric to a Fuchsian manifold. These
two types correspond to the beginning of the solution curve for the {\gse} \eqref{t-ge}.

In the case of $M^3$ being {\qf}, it admits at least one immersed area-minimizing {\is} (\cite{SY79, SU82}).
Furthermore, if the {\pc}s are less than one in magnitude, i.e., the case of {\af}, the incompressible {\ms} is unique.
Hence one can use {\ms}s to parametrize the space of {\af} manifolds within the {\qf} space (\cite{Uhl83, Tau04}),
and obtain important geometric and dynamical information about the {\af} manifolds, in terms of the geometry of
the {\ms}: for example, hyperbolic volume of the {\cvc} and the Hausdorff dimension of the limit set (\cite{HW09b}),
and Teichm\"uller distance between conformal infinities (\cite{GHW10}). There are also recent important work on
the applications of {\af} manifolds to mathematical physics (see for instance \cite{KS07,KS08}).
\subsection{The normal bundle}
Let $S \subset M^3$ be a {\mi} of $S$ into a {\htm} $M^3$, and $T^{\perp}S$ be its {\nb} in $M^3$. The Riemann
curvature tensor $R_{ijk\ell}$ on $M^3$ has six components, three of them satisfy curvature equation of the form:
\begin{equation} \label{ce}
R_{i3j3} = -g_{ij}, 
\quad 
i=(1,2), j=(1,2).
\end{equation}
From classical Riemannian geometry, the exponential map $\exp: T^{\perp}S \to M^3$ is a local diffeomorphism on
$S\times (-\epsilon, \epsilon) \subset T^{\perp}S$. Therefore, given first and {\sff}s on $S$, these equations \eqref{ce}
uniquely determine a {\hym} on $S\times (-\epsilon, \epsilon) \subset T^{\perp}S$. On the {\nb} $T^{\perp}S$, these
three equations can be reduced to a second order system for fixed $z \in S$:
\begin{equation*}
\frac12 \frac{\partial^2 g_{ij}(z,r)}{\partial r^2} -
\frac14 \frac{\partial g_{i\ell}(z,r)}{\partial r}g^{\ell k}\frac{\partial g_{kj}(z,r)}{\partial r} = g_{ij}(z,r),
\end{equation*}
whose solution can be written explicitly as follows (\cite{Uhl83}): for $(z,r) \in S \times (-\epsilon, \epsilon)$,
\begin{equation}\label{metric-r}
   g(z,r)=e^{2v(z)}[\cosh(r)\I+\sinh(r)
   e^{-2v(z)}A(z)]^{2}\ ,
\end{equation}
and the {\hym} on $T^{\perp}S$ is given as $ds^2 = g(z,r) |dz|^2 +dr^2$. Here the induced metric on $S$ is given by
$g_{ij}(z)=g_{ij}(z,0) =e^{2v(z)}\delta_{ij}$, where $v(z)$ is a smooth
function on $S$, and the {\sff} $A(z)=[h_{ij}]_{2\times{}2}$, where $z = x + \sqrt{-1}y$ is the conformal coordinates
on marked surface $(S, \sigma)$. With respect to these coordinates, the {\sff} of $S \subset M^3$ can be written as
\begin{equation} \label{h}
   h = h_{11}dx^2 + 2h_{12}dxdy+ h_{22}dy^2.
\end{equation}
The remaining three curvature equations are constraint equations for the first and {\sff} on $S$. Note that the {\ms}
$S$ has zero {\mc}, so we denote $\pm\lambda(z)$ the eigenvalues of $A(z)$, where $\lambda(z) \ge 0$.
They are the {\pc}s of $S$. In this case, $h_{11} = -h_{22}$ and two of the remaining three curvature equations are
the Codazzi equations: $R_{ijk3} = 0$. That is equivalent to say \eqref{h} becomes (\cite{LJ70}):
\begin{equation}\label{ha}
   h = Re(\alpha),
\end{equation}
for some $\alpha \in Q(\sigma)$, a {\hqd} on the marked {\RS} $(S,\sigma)$. Note that the {\hqd} $\alpha$ must
have zeros somewhere, or $|\alpha|dzd\bar{z}$ defines a smooth flat metric on $S$, violating the Gauss-Bonnet theorem.

The metric $g(z,r)$ might be singular if there are conjugate points of the exponential map. It is easy to verify that when
$\lambda(z) < 1$ for all $z \in S$, then the map $\exp$ has no conjugate point and the {\nb} $T^{\perp}S$  extends to
both infinities to become a complete {\htm} and $S$ is the only {\ms} (also embedded) in $T^{\perp}S$.
\subsection{The Gauss equation}
Five of six curvature equations for the {\mi} $S \subset M^3$ take the form of \eqref{ce} and \eqref{ha}. They determine
the metric in the {\nb} of $S$ and the {\sff} of $S$. The sixth equation is the {\gse} which describes the interaction between
the hyperbolic structure on $S$ and the ambient hyperbolic structure of $M^3$. Note that ours is slightly different from the
equation in [Theorem 4.2 \cite{Uhl83}] because of an obvious typo there. We recall from \eqref{ge}.
\begin{equation*}
\Delta u + 1 -e^{2u} - \frac{|\alpha|^2}{g_\sigma^2}e^{-2u} = 0,
\end{equation*}
here $g_\sigma |dz|^2$ is the {\hym} on $(S,\sigma)$, and the conformal factor between the induced metric $f(z)|dz|^2$
and $g_{\sigma}|dz|^2$ is determined by $f(z) = e^{2u(z)}g_\sigma(z)$. We also use $\Delta$ to denote the hyperbolic
Laplace operator on $(S,\sigma)$. This equation is similar to the prescribed scalar curvature equation studied by
Kazdan-Warner in \cite{KW74, KW75}.

The {\gse} is a consequence of two equivalent ways of describing the intrinsic curvature $K(z)$ induced by the {\mi}:
\begin{equation*}
K(z) = e^{-2u(z)}(-\Delta u(z) -1) = -1 -\lambda^2(z),
\end{equation*}
where the positive {\pc} $\lambda(z)$ is given by
\begin{equation}\label{pc}
  \lambda(z) = \frac{|\alpha|}{g_\sigma}e^{-2u},
\end{equation}
because of the equation \eqref{ha}.

We are particularly interested in a family of {\gse}s, corresponding to a ray $\alpha(t) = t\alpha \in Q(\sigma)$, as in
\eqref{t-ge}:
\begin{equation*}
\Delta u(t) + 1 -e^{2u(t)} - \frac{t^2|\alpha|^2}{g_\sigma^2}e^{-2u(t)} = 0.
\end{equation*}

\section{Proof of main theorems}
Our main theorems concern the solution curve for the family of {\gse}s \eqref{t-ge} and the geometry of the {\ms}s
corresponding to these solutions. For the parameter $t \ge 0$, we study the large values first, where we prove
solutions do not exist in \S 3.1. In the remaining sections, we focus on the range where solutions do exist, especially
Uhlenbeck's stable solutions, in \S 3.2. We then construct the mountain pass solutions and study their asymptotic
geometry in the remainder of the section.
\subsection{Non-existence result}

Let us first emphasize that equation~\eqref{t-ge} does not admit any solution for large value of $t$. More specifically,
the following theorem reveals some necessary properties for solutions to the {\gse}. In particular, it proves part (i) of
the Theorem 1.2.
\begin{theorem} \label{thm:Ciao}
\begin{enumerate}
\item[(a)]
Any solution $u$ to the {\gse} satisfies $u \leq 0$.
\item[(b)]
  For $t \geq  2\pi(2g-2) \left( \int_{S_{\sigma}}
                            \frac{|\alpha|}{|g_{\sigma}|} \right)^{-1}$, Problem~\eqref{t-ge} admits no solution.
\end{enumerate}
\end{theorem}
\begin{proof}
{\bf (a)} This is the consequence of the {\maxp}: At a maximum point $x_0$ of a solution $u$, apply the {\maxp} to the
equation~\eqref{ge} to obtain that  $0 \leq  1 - e^{2u (x_0)}$. Hence $ u \leq 0$, the conclusion follows.

\medskip
{\bf (b)} Let $(t,u)$ be a solution to ~\eqref{t-ge}, and $dA_{\sigma} = g_{\sigma}dzd\bar{z}$ be the hyperbolic area
element on $(S, g_{\sigma})$. On the one hand, by integrating equation~\eqref{t-ge} on $(S, g_{\sigma})$, 
and using that the area of $S$ is $2\pi(2g-2)$, 
we obtain
\begin{eqnarray}
 2\pi(2g-2) &=& \int_{S} e^{2 u }dA_{\sigma} + t^2 \int_{S}\frac{|\alpha|^2}{g_{\sigma}^2} e^{- 2 u}dA_{\sigma} \nonumber \\
               &>& t^2 \int_{S}\frac{|\alpha|^2}{g_{\sigma}^2} e^{- 2 u}dA_{\sigma}.  \label{eq:Giessen}
\end{eqnarray}
On the other hand, Cauchy-Schwarz inequality and the fact that $u \leq 0$ give
\begin{eqnarray}
   \left( \int_{S}{\frac{|\alpha|}{g_{\sigma}}}dA_{\sigma} \right)^2
   &=&
   \left( \int_{S} e^{u}\frac{|\alpha|}{g_{\sigma}} e^{-u}dA_{\sigma} \right)^2   \nonumber \\
   &\leq&
   \left( \int_{S} e^{2 u}dA_{\sigma}  \right) \left( \int_{S} \frac{|\alpha|^2}{g_{\sigma}^2} e^{-2 u}dA_{\sigma} \right)
   \nonumber \\
   &\leq&
   2\pi(2g-2) \left( \int_{S} \frac{|\alpha|^2}{g_{\sigma}^2} e^{-2 u}dA_{\sigma} \right) .
   \label{eq:Giessen2}
\end{eqnarray}
Relations \eqref{eq:Giessen} and \eqref{eq:Giessen2} imply, for any solution $(t,u)$, the following inequality holds
\begin{equation*}
 2\pi(2g-2) > t^2 \frac{\left( \int_{S_{\sigma}}{\frac{|\alpha|}{g_{\sigma}}} \right)^2}{2\pi(2g-2)},
\end{equation*}
and the conclusion follows.
\end{proof}

{\bf Remark}: {\it One can see above application of the Cauchy-Schwarz inequality as a comparison of two metrics
on {\TS}: the Teichm\"{u}ller metric and the {\WPm} (See the Proposition 2.4 of \cite{Mc00}). For the {\hqd}
$\alpha dz^2$, its Teichm\"{u}ller norm is $\|\alpha\|_{T} = \int_{S}|\alpha|dzd\bar{z}$, while its Weil-Petersson norm
is given by $\|\alpha\|_{WP} = \sqrt{\int_{S}\frac{|\alpha|^2}{g_{\sigma}^2}dA_{\sigma}}$}.

\medskip

We want to point out that if two solutions to \eqref{t-ge} do exist, then their geometric properties are different, i.e.,

\begin{theorem}
Assume the Gauss equation~\eqref{t-ge} admits two solutions $u_1 \not \equiv u_2$, and consider
their associated principal curvature $\lambda_i(x)= \frac{|\alpha|}{g_{\sigma}}  e^{-2 u_i}$. Then
$\lambda_1 \not \equiv \lambda_2$.
\end{theorem}
\begin{proof}
Assume $\lambda_1 \equiv \lambda_2$. Setting the intrinsic curvature $K (x) := -(1+ \lambda_1^2(x))$,
we see that the functions $u_i$ solves the equation
\begin{equation*}
-\Delta u_i - 1 = K(x) e^{2 u_i}  (i=1,2).
\end{equation*}
Hence we have
\begin{equation} \label{eq:Certo}
   -\Delta (u_2 - u_1) = K (x) \big( e^{2 u_2} - e^{ 2 u_1} \big),
\end{equation}
and therefore
\begin{equation} \label{eq:Ciao}
  \int_S |\nabla (u_1 - u_2)|^2 =  \int_{S} K(x) \big( e^{2 u_2} - e^{ 2 u_1} \big) (u_2 - u_1)   \,.
\end{equation}
Since $(e^s - e^t) (s-t) \geq  0$ for any $s,t \in \mathbb R$, while $K (x) \leq -1 < 0$,
equality~\eqref{eq:Ciao} implies that $u_2 - u_1 \equiv C$ for some constant $C \in \mathbb R$.
Clearly $K \not \equiv 0$, and therefore equality~\eqref{eq:Certo} implies $u_2 \equiv u_1$.
\end{proof}
Since the solutions in above theorem are in the same {\cs}, the Theorem 3.2 can also be seen as a
consequence of a comparison theorem for conformal metrics of negative curvature (\cite{Wol82}).

\subsection{Uhlenbeck's solution curve}
A solution curve to the {\gse} can be obtained from the implicit function theorem, as in \cite{Uhl83}. In this subsection,
we study this solution curve further, in anticipation of using it to construct our mountain pass solution.

Consider the nonlinear map $F: W^{2,2} (S) \times [0,\infty) \to L^2 (S)$ defined by
\begin{equation}\label{FF}
  F(u,t) = \Delta u + 1 -e^{2u } - \frac{t^2|\alpha|^2}{g_\sigma^2}e^{-2u },
\end{equation}
where $W^{2,k} (S)$ stands for the classical Sobolev space. At each $t \geq 0$ fixed, the {\lo} 
$L(u,t) : W^{2,2} (S) \to L^2 (S)$ associated to $F$ is given by  
\begin{equation}\label{L}
  L(u,t) =- \Delta  + 2 \left(e^{2u} - \frac{t^2|\alpha|^2}{g_\sigma^2}e^{-2u} \right),
\end{equation}
and the differential (Fr\'echet derivative) of $F$ is given by
\begin{equation}\label{dF}
  dF(u,t) (\dot{u},\dot{t})= -L\dot{u} - 2t\dot{t}\frac{|\alpha|^2}{g_\sigma^2}e^{-2u},
\end{equation}
where $(\dot{u}, \dot{t}) \in W^{2,2} (S) \times \mathbb R$.

The linear operator $L$ is geometrically meaningful, since its eigenvalues are closely related to the 
stability of the {\mi} by the following theorem of Uhlenbeck:
\begin{theorem}\cite{Uhl83}
A {\mi} with data $(\sigma,t\alpha)$ in any hyperbolic {\tm} $M^3$ is stable if and only if $L \ge 0$.
\end{theorem}

From the analytic point of view, the linear operators $L(u,t)$ in \eqref{L} and $dF(u,t)$ in \eqref{dF} are 
important in order to apply the implicit function theorem. 
When the {\lo} $L$ has all positive eigenvalues, the {\dio} $dF$ is onto. When
zero is the lowest eigenvalue for $L$, its kernel and cokernel are one-dimensional, and $dF$ is still onto.

\medskip

We easily see that there exists a constant $\tau := \tau (\sigma,\alpha)$, such that 
at any solution $F(u,t) =0$ with $t > \tau$,  the first eigenvalue of $L(u,t)$ is negative. Indeed 
since any solution satisfies $u < 0$ (as in Theorem~\ref{thm:Ciao}), from
\begin{equation*}
   e^{2u(t)} - \frac{t^2|\alpha|^2}{g_\sigma^2}e^{-2u(t)} < 1- \frac{t^2|\alpha|^2}{g_\sigma^2} ,
\end{equation*}
we readily see that for large $t$ the first eigenvalue of $L(u,t)$ is negative at any possible solution.
On the other hand, we have $F(0,0)=0$ and we immediately see that $L(0,0) >0$. Hence by applying the implicit function theorem, starting from this trivial solution, one obtains a smooth solution curve $\gamma$. More specifically we obtain
\begin{theorem}\cite{Uhl83} \label{thm:Karen}
There exists a smooth curve 
$$
   \gamma: [0,\tau_0] \to  W^{2,2} (S) \times [0,\infty)
   \qquad
   t \mapsto (u(t),t),
$$
such that 
\begin{enumerate}
\item[(a)]
$\gamma(0) = (0,0)$ and $F(\gamma(t) ) = 0$ for all $ t \in [0, \tau_0]$,
\item[(b)]
$L(u(t), t) > 0 $ for all $t \in [0, \tau_0)$,
\item[(c)]
${\rm Ker} \big( L(u(\tau_0), \tau_0) \big) \not = \{0 \}$. 
\end{enumerate}
%
\end{theorem}
Note that by standard regularity theory, the solution $u(t)$ obtained in the above theorem belongs to $C^{\infty}(S)$.

\subsection{Variational setting}
In the next three subsections, we prove our main theorems. Since we focus on finding additional solutions, we will not
use the parametrization of $t$ in these subsections.

In this subsection, we develop the variational setting of the problem. We will introduce a different but
equivalent norm to make use of the mountain pass theorem.

Setting $V(z) =  \frac{t^2 |\alpha|^2}{g_\sigma^2} \geq 0$, the {\gse} \eqref{ge} is given by:
\begin{equation*}
\Delta u + 1 -e^{2u} - V(z)e^{-2u} = 0,
\quad
 \hbox{ on } (S,\sigma),
\end{equation*}
which is the Euler-Lagrange equation of
\begin{equation} \label{eq:Identitat}
  I(u) :=
  {\frac12}\int_{S}|\nabla u|^2 - \int_{S} \left(u - {\frac{e^{2u}}{2}} \right) -\int_{S}V(z){\frac{e^{-2u}}{2}},
  \quad
  u \in H^1 (S).
\end{equation}
To derive compactness property, we introduce an equivalent norm on the Sobolev space and consider a new
functional whose set of critical points coincides with the one of~\eqref{eq:Identitat}.

To reach this goal, by choosing $\theta >2$, we define a function $F_1 \in C^{\infty} (\mathbb R)$ which satisfies
\begin{equation*}
  F_1 (s) := \left\{
  \begin{array}{cl}
    s - \frac{1}{2} e^{2 s}      &\hbox{ if } s \leq 0
    \vspace{5pt} \\
    - s^{\theta}                        &\hbox{ if } s > 1
  \end{array} \right.
  \qquad
  F_1' (s) < 0 \quad \forall s > 0 \,,
\end{equation*}
and $F_2 \in C^{\infty} (\mathbb R)$ defined as
\begin{equation*}
  F_2 (s) := \left\{
  \begin{array}{cl}
    \frac{1}{2} ( s^2  + e^{-2 s} )   & \hbox{ if } s \leq 0
    \vspace{3pt} \\
    0                                 & \hbox{ if } s > 1
  \end{array} \right.
  \qquad
  F_2' \leq 0 .
\end{equation*}

Setting $f_i (s):= F_i' (s)$ ($i=1,2$), we explicitly have
\begin{equation*}
  f_1 (s) := \left\{
  \begin{array}{cl}
    1 - e^{2s}                     & \hbox{ if } s \leq 0 \\
    -\theta s^{\theta -1}          & \hbox{ if } s > 1
  \end{array} \right.
  \qquad
  f_2 (s) := \left\{
  \begin{array}{cl}
    s - e^{-2 s}        & \hbox{ if } s \leq 0 \\
    0                   & \hbox{ if } s > 1
  \end{array}, \right.
\end{equation*}
which coincide for $ s \leq 0$ with the nonlinearities arising in the Gauss equation~\eqref{ge},
and have the following property:
\begin{equation} \label{eq:fProp}
  f_1 (s) < 0 \quad \forall s > 0,
  \qquad
  f_2 (s) \leq 0
  \quad \forall s \in \mathbb R .
\end{equation}
\begin{pro}\label{NG}
The {\gse} ~\eqref{ge} is equivalent to the new equation
\begin{equation} \label{eq:NewGauss}
  - \Delta u + V(z) u - \big( f_1 (u) +  V(z) f_2 (u) \big) = 0  \,.
\end{equation}
\end{pro}
\begin{proof}
The {\maxp} easily implies that $u < 0$. From the explicit formulas of $f_1(u)$ and $f_2(u)$ for $u < 0$, it is easy
to verify that
\begin{eqnarray*}
-\Delta u + V(z)u -  f_1 (u) - V(z) f_2 (u) &=& - \Delta u -1 +e^{2u} + V(z)e^{-2u} \\
&=& 0.
\end{eqnarray*}
Now we apply the {\maxp} on ~\eqref{eq:NewGauss}, and make use of the properties ~\eqref{eq:fProp}, we easily
see that the solutions of ~\eqref{eq:NewGauss} are negative. Hence the sets of solutions of~\eqref{ge} and
~\eqref{eq:NewGauss} coincide.
\end{proof}

We will work with ~\eqref{eq:NewGauss}, since it admits a variational formulation that satisfies compactness
property. More specifically, we consider the usual Sobolev space
$$
   H^1 (S) := \{ u \in L^2 (S) \, \colon \, \nabla u \in L^2 (S) \},
$$
endowed with the following inner products
$$
   \langle f, g\rangle := \int_{S} \Big\{ \nabla f \nabla g + fg \Big\}.
\qquad
  \langle f, g \rangle_V :=
  \int_{S} \Big\{ \nabla f \nabla g + V(z) fg \Big\},
$$
and denote $\| \cdot \|_{H^1}$, $\| \cdot \|_V$ as their associated norms, respectively.

\begin{lem} \label{lem:Equivalent}
The norms $\| \cdot \|_{H^1}$, $\| \cdot \|_V$ are equivalent.
\end{lem}
\begin{proof}
Since $V \in L^{\infty} (S)$, we clearly have $\| \cdot \|_V \leq C \| \cdot \|_{H^1}$.

Setting $\bar u:= \frac{1}{|S|} \int_S u$, where $|S|$ is the hyperbolic area of the surface, we have
\begin{equation} \label{eq:Poincare}
  \| u - \bar u \|_{L^2} \leq  C \| \nabla u \|_{L^2}
\end{equation}
from the Poincar\'e inequality. Furthermore,
\begin{eqnarray*}
  \int_S V(z) |\bar u|^2
  &\leq&
  2 \int_S
  \Big\{ V(z) |\bar u - u|^2 + V(z) u^2 \Big\}
  \\
  &\leq&
  C \int_S \Big\{ |\bar u - u|^2 + V(z) u^2 \Big\}
  \\
  &\leq&
  C \int_S \Big\{ | \nabla u |^2 + V(z) u^2 \Big\}.
\end{eqnarray*}
Therefore, since $\int_S V >0$ we get
\begin{equation} \label{eq:Cagliari}
  |\bar u|^2  \leq C \Big( \int_S V(z) \Big)^{-1}
  \int_S \Big\{ | \nabla u |^2 + V(z) u^2 \Big\}
  \leq C \| u \|_V .
\end{equation}
Using then ~\eqref{eq:Poincare} and~\eqref{eq:Cagliari}, we conclude
$$
  \| u \|_{L^2}
  \, \leq \,
  \| u - \bar u \|_{L^2} + \| \bar u \|_{L^2}
  \, \leq \,
  C \| u \|_V \,.
$$
This immediately implies $\| u \|_{H^1} \leq C \| u \|_V$, which completes the proof.
\end{proof}
Now we can define the associated functional for the equation ~\eqref{eq:NewGauss}. In the Hilbert space
$H^1 (S)$, the functional
\begin{equation}\label{F}
    \mathcal{F} (u) :=  \frac{1}{2}
     \int_S \big\{  |\nabla u|^2 + V(z) u^2 \big\}
   - \int_S \big\{ F_1 (u) + V(z) F_2 (u)   \big\} ,
   \quad
   u \in H^1 (S) ,
\end{equation}
is by the Moser-Trudinger inequality well defined, of class $C^1$, and its critical points are weak solutions of
~\eqref{eq:NewGauss}. In this functional setting we will be able to use a minimax argument to derive a
second solution to the {\gse}.

\subsection{Mountain pass structure}
In this subsection, we show that the functional $\Fcal$ exhibits a mountain pass geometry.
We need the following compactness property,
\begin{theorem} \label{PS}
The functional $\mathcal{F}$ satisfies the Palais-Smale condition, i.e., any sequence $\{u_n\}$ in $H^1 (S)$
satisfying,
\begin{equation} \label{eq:PS}
  |\mathcal{F} ( u_n ) | \leq C ,
  \qquad
  \|\mathcal{F} '( u_n) \|_{H^{-1}} \to 0 ,
\end{equation}
admits a subsequence converging strongly in $H^1 (S)$.
\end{theorem}
\begin{proof}
Consider the exponent $\theta > 2$ appearing in the definition $F_1$. We claim that
$$
  F_1 (s)
  \, \leq \,
  \frac{s}{\theta}  f_1 (s) + O(1),
  \qquad
  F_2 (s)
  \, \leq \,
  \frac{s}{\theta} f_2 (s)  + O(1) \,.
$$
Indeed from our definition of $F_i$ ($i=1,2$), note first that for all $s \in (1,\infty)$
$$
  F_1(s) = \frac{1}{\theta} s f_1(s) ,
  \qquad
  F_2 (s) = s f_2 (s) =0.
$$
Secondly, for $s \in [0,1]$ we obviously have:
$$
   F_i (s) \leq s f_i (s) + O(1) \quad (i=1,2).
$$
Thirdly, for $s <0$ we note that
$$
  F_1 (s) \leq s \leq \frac{s}{\theta}
  \, = \,
  \frac{s}{\theta} (1 -e^{2s}) + O(1)
  \, = \,
  \frac{s}{\theta} f_1(s) + O(1)  ,
$$
$$
  F_2 (s)
  \, \leq \,
  e^{-2s}
  \, \leq \,
  \frac{1}{\theta} (-s) e^{-2s} + O(1)
  \, \leq \frac{s}{\theta} f_2 (s) + O(1) \,.
$$
Hence, we have:
\begin{equation} \label{eq:ARCondition}
   F_1 (s) + V(z) F_2 (s)
   \, \leq \,
   \frac{s}{\theta}  \Big( f_1 (s) + V(z) f_2 (s) \Big)  + O(1) \,.
\end{equation}

To proceed with the proof, given a sequence $\{u_n\}$ satisfying~\eqref{eq:PS}, we prove that it is bounded.
Condition~\eqref{eq:PS} implies that
\begin{equation}
  \frac{1}{2} \| u_n \|_V^2 - \int_S \Big\{ F_1 (u_n) + V(z) F_2 (u_n) \Big\} = O(1)
  \label{eq:PS1}
\end{equation}
\vspace{-2mm}
\begin{equation}
  \|u_n \|_V^2 - \int_S  u_n \Big\{ f_1 (u_n) + V(z) f_2 (u_n)  \Big\}   = o(1)
  \label{eq:PS2}.
\end{equation}
Using successively~\eqref{eq:PS1}, \eqref{eq:ARCondition} and~\eqref{eq:PS2}, we deduce
\begin{eqnarray*}
  \frac{1}{2} \|u_n \|_V^2
  &\leq&
  C + \int_S \Big\{ F_1 (u_n) + V(z) F_2 (u_n) \Big\}
  \\
  &\leq&
  C + \frac{1}{\theta} \int_S u_n \Big\{ f_1 (u_n) + V(z) f_2 (u_n)  \Big\}  \\
  &=&
  O(1) + \frac{1}{\theta} \|u_n\|_V^2 \,.
\end{eqnarray*}

Therefore for some $\theta > 2$, we have
$$
  \left( \frac{1}{2} - \frac{1}{\theta} \right) \|u _n \|_V^2 = O(1).
$$
So $\|u _n \|_V = O(1)$, and therefore we have $u_n \rightharpoonup \hat u$ weakly in $H^1 (S)$ (up to a subsequence).

Furthermore, we have
\begin{equation} \label{eq:CiSaraSperanza}
  \| u_n - \hat u \|_V^2=
  \int_S \Big\{ f_1 (u_n) (u_n - \hat u) + V(z) f_2 (u_n) (u_n - \hat u) \Big\}
  + o(1) .
\end{equation}
By Lebesgue dominated convergence, we readily have $f_1 (u_n) \to f_1 (\hat u)$ strongly in $L^2 (S)$. Moreover, using
the Moser-Trudinger inequality, it is known that the map
$$
  H^1 (S) \to L^2
  \quad
  u \mapsto e^u
$$
is compact. Therefore we also have $f_2 (u_n) \to f_2 (\hat u)$ strongly in $L^2(S)$. Hence~\eqref{eq:CiSaraSperanza}
implies that $\| u_n - \hat u \|_{V} = o(1)$.
By Lemma \ref{lem:Equivalent}, we conclude that $\| u_n - \hat u \|_{H^1} = o(1)$.
\end{proof}

We now show that the functional $\Fcal$ admits mountain pass type solutions
\begin{pro}\label{TwoSolutions}
For each $t \in (0, \tau_0)$ where $\tau_0$ is defined in Theorem~\ref{thm:Karen},
the functional $\mathcal{F}$ admits at least two solutions.
\end{pro}
\begin{proof}
Using the definition $F_i$ $(i=1,2)$, we see that at each point $u \in H^1 (S)$ with $u \leq 0$ we have
$I (u) =  \mathcal{F}(u)$ and the second derivative at $u$ satisfy
\begin{equation} \label{eq:PJ}
   I''_{(u)} (\xi,\xi) = \mathcal{F}''_{(u)} (\xi,\xi)
  \quad \forall \xi \in H^1 (S) \,,
\end{equation}
and this bilinear form is explicitly given by
$$
  I''_{(u)} (\xi,\xi)
  =
  \int_S |\nabla  \xi|^2 +
  2 \int_S  \left(  e^{2u} -  t^2 \frac{|\alpha(x)|^2}{g_{\sigma}^2} e^{-2 u} \right) \xi^2 \,.
$$
Now at the stable solution $u (t)$ obtained by Uhlenbeck's Theorem~\ref{thm:Karen}, we note that for $t \in (0, \tau_0)$, 
standard results show that the first eigenvalue of the linearized operator $L (u(t),t)$ is given by the infimum 
of the Rayleigh quotient $ \frac{I''_{(u (t) )} (\xi,\xi)}{\|\xi \|^2_{H^1}}$ with $\xi \in H^1(S) \setminus \{0 \}$. 
Therefore for each $t \in (0, \tau_0)$, we deduce
\begin{equation} \label{eq:DoIt}
   I''_{(u (t) )} (\xi,\xi) \geq C \|\xi \|^2_{H^1}
   \quad \forall \xi \in H^1 (S) ,
\end{equation}
i.e. $u (t) $ is a local minimizer of the functional $I$. Since $u (t) <0$, equality~\eqref{eq:PJ}, with~\eqref{eq:DoIt}
and Lemma \ref{lem:Equivalent} imply
$$
  \mathcal{F}_{(u (t) )} (\xi, \xi) \geq \| \xi \|_V \,.
$$
Hence there exists a ball $B(u (t) , r)$ in the Hilbert space $H^1 (S)$ such that
\begin{equation} \label{eq:Mountain1}
  \inf_{u \in \partial B(u (t) , r) } \mathcal{F} (u) \geq \mathcal{F} (u (t) )  \,.
\end{equation}
Furthermore, take $w$ to be a constant negative function, we easily see that
$$
  \mathcal{F} (w) =  - \Big( w - \frac{1}{2}e^{2 w} \Big) |S| -e^{-2 w} \int_S V(z),
$$
namely
$$
  \lim_{w \to - \infty} \mathcal{F} (w) = - \infty .
$$
Hence, there exists $w \in H^1 (S)$ such that
\begin{equation} \label{eq:Mountain2}
   w \not \in B(u (t), r),
   \qquad
   \mathcal{F} (w) < \mathcal{F} (u (t)) \,.
\end{equation}
Since the Palais-Smale condition is satisfied, conditions \eqref{eq:Mountain1} and \eqref{eq:Mountain2} allow to apply the
Mountain-Pass Theorem of Ambrosetti-Rabinowitz \cite{AR73}.
\end{proof}
\begin{proof}[Proof of Theorem 1.2:] 
Now the theorem follows from the Proposition~\ref{TwoSolutions} and the Proposition~\ref{NG}.
\end{proof}
\subsection{The asymptotic geometry}
We show now that as $t \to 0$ the mountain pass solutions blow-up.
\begin{pro}\label{alt}
Let $(t_n, u_n)$ be a sequence of critical points with $t_n \to 0$.
Then, along a subsequence, the following alternative holds
\begin{enumerate}
\item[(i)]
 $\| u_n \|_{H^1} \to 0$;
\item[(ii)]
 or
 $\| u_n \|_{H^1} + \|u_n\|_{\infty} \to \infty$.
\end{enumerate}
\end{pro}
\begin{proof}
We have two possibilities: either $\| u_n \|_{H^1} = O(1)$ or $\| u_n \|_{H^1} \to \infty$ (up to a subsequence).

\medskip

{\bf Case 1:}
Assume $\| u_n \|_{H^1} = O(1)$. \\
Then $u_n \rightharpoonup \bar u$, and by the Moser-Trudinger inequality
we know that $e^{\pm u_n} \to e^{\pm \bar u}$ in $L^2 (S)$. Since for each $\xi \in C^{\infty} (S)$ we have
$$
  \int_{S} \nabla u_n \nabla \xi = \int_{S} \xi - \int_S \{ e^{u_n} + t_n^2 |\alpha (x)|^2 e^{-u_n} \} \xi ,
$$
for $n \to \infty$ we get
$$
  \int_{S} \nabla {\bar u} \nabla \xi = \int_{S} \xi - \int_S e^{ {\bar u} } \xi
$$
Therefore $- \Delta  {\bar u} = 1 - e^{ {\bar u} }$, which implies $\bar u = 0$.

\medskip

{\bf Case 2:} Assume $\| u_n \|_{H^1 (S)} \to  \infty$.\\
From the identity
\begin{equation} \label{eq:Lemonade}
  \int_{S} |\nabla u_n |^2 = \int_{S} u_n - \int_S \{ e^{u_n} + t_n^2{\frac{|\alpha|^2}{|g_{\sigma}|^2}} e^{-u_n} \} u_n,
\end{equation}
we immediately see that $\| u _n \|_{\infty} \to \infty$. Indeed, if this is not the case we would have
$\|u_n\|_L^2 = O(1)$ and \eqref{eq:Lemonade} would imply $\int_{S} |\nabla u_n |^2 \leq C$.
\end{proof}
\begin{proof}[Proof of Theorem 1.3:] Recall that the positive {\pc} of the {\mi} given by $u_n(t_n)$ is given by
\begin{equation*}
\lambda(t_n) = {\frac{t_ne^{2u_n}|\alpha|}{g_{\sigma}}}.
\end{equation*}
We use this to examine our options from the above Proposition \eqref{alt}. In the case one, the {\pc}s are small for $t$ near
zero. Since in a neighborhood of $(0,0)$, the functional admits a unique branch of solution $(t,u_t)$, we conclude
that $u_n$ coincides with the solution $u_t$ obtained by the implicit function theorem from Uhlenbeck's theorem. Since the
{\pc}s are small (less than one in absolute value), the normal bundle is an {\af} {\tm}.

\medskip

While in the case two, we find $|\lambda(t_n)| \to \infty$ as $t_n \to 0$. Indeed assume on the contrary that $\lambda (t_n)$
stays uniformly bounded as $t_n \to 0$. Let $x_n$ be such that $ \min u = u_{t_n} (x_n) \to - \infty$.
By the minimum principle we deduce that
$$
    1 - (1 + \lambda (t_n)^2) e^{2 u_{t_n} (x_n)}  \leq 0 .
$$
If $\lambda(t_n)$ stays uniformly bounded we get a contradiction. This proves Theorem 1.3.
\end{proof}

\bibliographystyle{amsalpha}
\bibliography{ref-gauss}
\end{document}